\documentclass[12pt, reqno]{amsart}
\usepackage{amsmath}
\usepackage{amsthm}
\usepackage{amssymb}
\usepackage{amsrefs}
\usepackage{mathrsfs}
\usepackage[usenames]{color}
\usepackage[utf8x]{inputenc}
\usepackage[normalem]{ulem}
 \newtheorem{thm}{}[section]
 \newtheorem{theorem}[thm]{Theorem}
 
 \newtheorem{lemma}[thm]{Lemma}
 \newtheorem{proposition}[thm]{Proposition}
 \theoremstyle{definition}

 \numberwithin{equation}{section}
 \allowdisplaybreaks

\newcommand{\NN}{\ensuremath{\mathbb{N}}}

\newcommand{\FF}{\ensuremath{\mathbb{F}}}
\newcommand{\Sym}{\ensuremath{\mathbb{S}}}
\newcommand{\XX}{\ensuremath{\mathbb{X}}}
\newcommand{\YY}{\ensuremath{\mathbb{Y}}}
\newcommand{\xx}{\ensuremath{\mathbf{x}}}
\newcommand{\yy}{\ensuremath{\mathbf{y}}}

\newcommand{\ee}{\ensuremath{\mathbf{e}}}
\newcommand{\sss}{\ensuremath{\mathbf{s}}}
\newcommand{\SSB}{\ensuremath{\mathcal{S}}}
\newcommand{\EE}{\ensuremath{\mathcal{E}}}
\newcommand{\ww}{\ensuremath{\mathbf{w}}}
\newcommand{\vv}{\ensuremath{\mathbf{v}}}
\newcommand{\WW}{\ensuremath{\mathcal{W}}}
\newcommand{\OO}{\ensuremath{\mathcal{O}}}
\newcommand{\Id}{\ensuremath{\mathrm{Id}}}
\newcommand{\BB}{\ensuremath{\mathcal{B}}}
\newcommand{\DD}{\ensuremath{\mathcal{D}}}
\newcommand{\II}{\ensuremath{\mathrm{I}}}
\newcommand{\supp}{\operatorname{supp}}
\newcommand{\spn}{\operatorname{span}}

\begin{document}

\title[Garling sequence  spaces are not superreflexive]{Non-superreflexivity of Garling sequence spaces and applications to the existence of special types of conditional bases} 

\author[F. Albiac]{Fernando Albiac}\address{Mathematics Department--InaMat\\ Universidad P\'ublica de Navarra\\
 Pamplona 31006\\  Spain} 
\email{fernando.albiac@unavarra.es}
 
\author[J. L.  Ansorena]{Jos\'e Luis  Ansorena}\address{Department of Mathematics and Computer Sciences\\
Universidad de La Rioja\\
 Logro\~no 26004\\ Spain} 
\email{joseluis.ansorena@unirioja.es}
 
\author[S. J. Dilworth]{Stephen J. Dilworth}\address{Department of Mathematics\\ University of South Carolina \\ Co\-lum\-bia
SC 29208 \\ USA} 
\email{dilworth@math.sc.edu}
 
 \author[Denka Kutzarova]{Denka Kutzarova}
 \address{ 
 Department of Mathematics University of Illinois at Urbana-Cham\-paign  \\ 
 Urbana, IL 61801 \\ USA \\and
Institute of Mathematics\\ Bulgarian Academy of Sciences\\
 Sofia\\ Bulgaria.
 } \email{denka@math.uiuc.edu}



\begin{abstract} In this paper we settle in the negative the problem of the superreflexivity of Garling sequence spaces by showing that they contain  a complemented subspace isomorphic to a non superreflexive mixed-norm sequence space. As a by-product of our work, we give applications of this result  to the study of conditional Schauder bases and conditional almost greedy bases in this new class of Banach spaces.
\end{abstract} 

\thanks{F. Albiac and J. L. Ansorena acknowledge the support of the grant MTM2014-53009-P (MINECO, Spain). F. Albiac was also  supported by the grant    MTM2016-76808-P (MINECO, Spain). S. J. Dilworth was supported by the National Science Foundation under Grant Number  DMS-1361461. S. J. Dilworth and Denka Kutzarova were supported by the Workshop in Analysis and Probability at Texas A\&M University in 2017. }

\maketitle

\section{Introduction and background}\label{Introduction}

 \noindent   Suppose $1\le p<\infty$ and let  $\ww=(w_j)_{j=1}^\infty$ belong to the set of weights
 \[
 \WW:=\left\{(w_j)_{j=1}^\infty\in c_0\setminus\ell_1:1=w_1\geq w_2 \ge\cdots w_j \ge w_{j+1} \ge\cdots>0\right\}.\]
 The \textit{Garling sequence space}, denoted  $g(\ww,p)$, is the Banach space consisting of all  scalar sequences $f=(a_j)_{j=1}^\infty$ such that
\[
 \Vert  f  \Vert_{g(\ww,p)} = \sup_{\phi\in\OO } \left( \sum_{j=1}^\infty |a_{\phi(j)}|^p w_j \right)^{1/p}<\infty,
 \] 
 where  $\OO$ denotes  the set of all increasing functions from $\NN$ to $\NN$.

The study of the isomorphic structure of these  spaces, which generalize an example of Garling from \cite{Garling1968}, has been recently initiated in  \cite{AAW1}. For expositional ease, we have gathered in Theorem~\ref{PreviousResults}  a few geometric properties of $g(\ww,p)$ that will help the reader to contextualize the results herein.

\begin{theorem}[see \cites{AAW1,  AALW}]\label{PreviousResults}  Let  $1\le p<\infty$ and $\ww\in \WW$. Then:
\begin{itemize}

\item[(a)]  The unit vector system $(\ee_j)_{j=1}^\infty$  is a $1$-subsymmetric basis of $g(\ww,p)$  which is not symmetric.

\item[(b)] Any subsymmetric basis of $g(\ww,p)$ is equivalent to $(\ee_j)_{j=1}^\infty$.

\item[(c)] $g(\ww,p)$ is reflexive if and only if $p>1$.

\item[(d)] For every  $\varepsilon>0$ and every infinite dimensional subspace  $\YY$ of $g(\ww,p)$ there is a further subspace $\mathbb Z\subseteq \YY$ that is $(1+\varepsilon)$-isomorphic to $\ell_p$ and 
$(1+\varepsilon)$-complemented in $g(\ww,p)$.

\end{itemize}
\end{theorem}

Garling sequence spaces can be regarded as the subsymmetric counterpart of Lorentz sequence spaces $d(\ww,p)$, consisting of all scalar sequences $f=(a_n)_{n=1}^\infty$ such that  \[
 \Vert  f\Vert_{d(\ww,p)}=\sup_{\sigma\in\Pi}  \left(\sum_{n=1}^{\infty} |a_{\sigma(n)}|^p w_n\right)^{1/p}<\infty,
  \]
where $\Pi$ is the set of permutations of $\NN$.
The  spaces $d(\ww,p)$ were thoroughly investigated by Altshuler, Casazza and Lin in the early 1970's in the papers \cites{ACL73,CL74}. Subsequently, Altshuler showed in   \cite{Altshuler1975} that $d(\ww,p)$ is superreflexive if and only if $p>1$ and the weight $\ww$ is \textit{regular}, i.e.,
\[
 \sup_m \frac{1}{m w_m} \sum_{j=1}^m w_j<\infty.
 \] 

As for  the spaces $g(\ww,p)$, the first attempt to determine whether or not  they were superreflexive was undertaken in \cite{AAW2}. Using nonlinear tools from approximation theory such as the fundamental function of the canonical basis,
  the authors showed that $g(\ww,p)$ fails to be superreflexive if the weight $\ww$ is not regular.  
 
In this note  we adopt a radically different, intrinsic approach, in the sense that    the methods we use fall within the linear category.  To be precise, in Section~\ref{Main} we  study in detail the complemented subspaces of $g(\ww,p)$, and having done the groundwork we settle in the negative the problem of the superreflexivity of Garling sequence spaces in all cases by proving the following theorem:

\begin{theorem}\label{NoSR} The Banach space $g(\ww,p)$ is not superreflexive for any $1\le p<\infty$ and any $\ww\in\WW$.
\end{theorem}

Hence, while comparing Theorem~\ref{PreviousResults} with the corresponding results  in \cite{ACL73}
 reflects the fact that  Garling sequence spaces behave to some extent similarly to Lorentz sequence spaces replacing symmetry with subsymmetry, Theorem~\ref{NoSR} exhibits an important structural difference between these two types of  spaces. 

 In Section~\ref{ConditionalBases} we will apply the  results from Section~\ref{Main} to  investigate the existence of conditional bases  with some special features in $g(\ww, p)$, which bridges our results  with the theory of greedy approximation in Banach spaces.
 
 Standard Banach space notation and terminology are used throughout (see \cite{AlbiacKalton2016}). For clarity, however, we record the notation that is used most heavily.
We write   $\FF$ for the real or complex scalar field. 
Given a set of indices $I$, we denote by
 $(\ee_i)_{i \in I}$  the unit vector system of $\FF^I$, i.e.,  $\ee_i=(\delta_{i,j})_{j\in I}^\infty$, were
$\delta_{i,j}=1$ if $i=j$ and $\delta_{i,j}=0$ otherwise.
If $a_j$ are elements in a vector space, we will use the convention $\sum_{j=1}^0 a_j=0$.

Given families of non-negative real numbers $(\alpha_i)_{i\in I}$ and $(\beta_i)_{i\in I}$ and $0<C<\infty$ the symbol $\alpha_i\lesssim_C \beta_i$ for $i\in I$ means that $\alpha_i\le C \beta_i$ for all $i\in I$, while $\alpha_i\approx_C \beta_i$ for $i\in I$ means that $\alpha_i\lesssim_C \beta_i$ and $\beta_i\lesssim_C \alpha_i$ for $i\in I$. 
Now suppose $(\xx_n)_{n=1}^\infty$ and $(\yy_n)_{n=1}^\infty$ are basic sequences in $\XX$ and $\YY$, respectively.  
We say that $(\yy_n)_{n=1}^\infty$ $C$-\textit{dominates} $(\xx_n)_{n=1}^\infty$ and write $(\xx_n)_{n=1}^\infty\lesssim_C(\yy_n)_{n=1}^\infty$ if
$
\left\Vert\sum_{n=1}^\infty a_n\xx_n\right\Vert_X\lesssim_C\left\Vert\sum_{n=1}^\infty a_n\yy_n\right\Vert_Y
$ 
for all   $(a_n)_{n=1}^\infty\in c_{00}$.
Whenever $(\xx_n)_{n=1}^\infty\lesssim_C(\yy_n)_{n=1}^\infty$ and $(\yy_n)_{n=1}^\infty\lesssim_C(\xx_n)_{n=1}^\infty$, we say that $(\xx_n)_{n=1}^\infty$ and $(\yy_n)_{n=1}^\infty$ are $C$-\textit{equivalent}, and write $(\xx_n)_{n=1}^\infty\approx_C(\yy_n)_{n=1}^\infty$.
In all the above cases, when the value of the constant $C$ is irrelevant, we simply drop it from the notation.

 The norm of a linear operator from $T$ from a Banach space $\XX$ into a Banach space $\YY$ 
is denoted by $\Vert T\colon \XX\to\YY\Vert$.
Given a basis  $\BB=(\xx_j)_{j=1}^\infty$ for a Banach space $\XX$ the \textit{support} of  $f=\sum_{j=1}^\infty a_j\, \xx_j\in\XX$ 
with respect to $\BB$ is the set
$\supp(f)=\{j\colon a_j\not=0\}$. 
The \textit{coordinate projection}  on a set $A$ will be denoted by $S_A[\BB,\XX]$ or, if $\BB$ and $\XX$ are clear from context, $S_A$.
Given $1\le p <\infty$, $\left(\oplus_{n=1}^\infty \XX_n\right)_{n=1}^\infty$ denotes the direct sum in the 
$\ell_p$ sense of the sequence of Banach spaces $(\XX_n)_{n=1}^\infty$. 

More specialized notions from Banach space theory or approximation theory will be introduced as needed.

\section{Complemented subspaces of Garling sequence spaces}\label{Main}
\noindent
Theorem~\ref{NoSR} will be a consequence of the following result, which trivially implies that $\ell_\infty$ is finitely representable in any  Garling sequence space.

\begin{theorem}\label{ComplementedBesov}Let $1\le p<\infty$ and $\ww\in\WW$. For each $\varepsilon>0$ there is a  sublattice $Z\subseteq g(\ww,p)$
that is $(1+\varepsilon)$-lattice complemented in $g(\ww,p)$ and $(1+\varepsilon)$-lattice isomorphic to 
$(\oplus_{n=1}^\infty \ell_\infty^n)_p$.
\end{theorem}

In turn, the proof of  Theorem~\ref{ComplementedBesov} relies on Lemmas~\ref{Lemma:GlidingHumpWeight} and 
~\ref{Lemma:3} below. The former  is elementary and exhibits a ``gliding hump'' behaviour of the weights in $\WW$. 
\begin{lemma}\label{Lemma:GlidingHumpWeight} Let $(w_j)_{j=1}^\infty$ be a non-increasing sequence of positive numbers with
$\sum_{j=1}^\infty w_j=\infty$. 
For every $m\ge 0$ we have
\[
\lim_k\frac{\sum_{j=m+1}^{m+k} w_j}{\sum_{j=1}^{k} w_j}=1.
\]
\end{lemma}
\begin{proof}Since $\ww$ is non-increasing,
\[
A:=\limsup_k\frac{\sum_{j=m+1}^{m+k} w_j}{\sum_{j=1}^{k} w_j}\le 1,
\]
and, since $\lim_k \sum_{j=1}^k w_j=\infty$,
\[
B:=\liminf_k\frac{\sum_{j=m+1}^{m+k} w_j}{\sum_{j=1}^{k} w_j}
=\liminf_k\frac{\sum_{j=1}^{m+k} w_j}{\sum_{j=1}^{k} w_j}
\ge 1.
\]
Since $B\le A$, we obtain  $A=B=1$.\end{proof}

 In order to state and prove the following lemmata, it is convenient to set some notation. Given a tuple $f=(a_j)_{j=1}^k$ 
 and $m\in\NN\cup\{0\}$, let us define the sequence $\II_m(f)$ in $c_{00}$  by 
\[
\II_m(f)=
(\underbrace{0,\dots,0}_{m},a_1, \dots, a_j, \dots, a_k, 0,\dots,0,\dots).
\]
Note that
$\II=\II_0$ is the natural embedding of $\cup_{k=1}^\infty \FF^k$ into $\FF^\NN$. We define
\[
\Vert f \Vert_g:=\Vert \II(f)\Vert_g, \quad f\in \cup_{k=1}^\infty\FF^k.
\] 
By the  $1$-subsymmetry of the unit vector basis we have $\Vert \II_m(f)\Vert_g=\Vert f\Vert_g$ for every tuple $f$ and every $m\in\NN$.

Given two tuples $f$ and $g$, the symbol $(f,g)$ (elsewhere, $f^\smallfrown g$) denotes its concatenation.

\begin{lemma}\label{Lemma:1}  Let $1\le p<\infty$ and $\ww\in\WW$.
Given $1<t<\infty$ and  tuples $f_{1}$  and $f_{2}$ with $\Vert f_{1}\Vert_{g} <t$, there  is a tuple $h$ such that $\Vert (h,f_{1})\Vert_g <t$ and $\Vert (f_{2},h)\Vert_{g} \ge \left(\Vert f_{2} \Vert_{g}^p +1\right)^{1/p}$. Moreover, $h$ can be chosen to be a constant-coefficient  $k$-tuple with $k$  as large as wished.
\end{lemma}

\begin{proof} Pick $\max\{ 1, \Vert f_{1}\Vert_g\} <s<t$.
Let $f_{1}=(a_j)_{j=1}^n$ and  $f_{2}=(b_j)_{j=1}^m$. Put $a_j=0$ for $j>n$ and define for every non-negative integer $k$
\[
v_k=\sup_{\phi\in\OO} \sum_{j=1}^\infty |a_{\phi(j)}|^p w_{j+k}.
\]
Since $\ww$ is non-increasing we have 
$v_k\le \Vert f_{1}\Vert_g^p < s$ and 
\[
v_k\le  \left(\sum_{i=1}^n |a_i|^p\right) w_{1+k}.
\]
 Hence, $\lim_k v_k=0$ since $\ww\in c_0$.
For any $k\in\NN$ put
\[
\alpha_k=\frac{s-v_k}{\sum_{j=1}^k w_j}\in (0,\infty).
\]
By Lemma~\ref{Lemma:GlidingHumpWeight}, $\lim_k \alpha_k \sum_{j=m+1}^{m+k} w_j=s$, and since $\ww\notin\ell_1$,  $\lim_k\alpha_k=0$. We infer that there is $k\in\NN$, which can be chosen larger than a given integer, such that 
\begin{itemize} 
\item[(i)] $\displaystyle \alpha_k \sum_{j=m+1}^{m+k} w_j \ge 1$ and
\item[(ii)]$\displaystyle \alpha_i\ge \alpha_k$ whenever $1\le i \le k$. 
\end{itemize}
Indeed, given $k_0\in\NN$ there is $k_1\ge k_0$ such that (i) holds for every $k> k_1$. Let $k_2>k_1$ be such that
\[
\alpha_{k_2}<\min_{1\le i \le k_1} \alpha_i.
\]
If we pick $k\in\{ 1,\dots, k_2\}$ where $\min_{1\le i \le k_2} \alpha_i$ is attained then (ii) holds, 
 and  $k>k_1$.

Let $h$ be the constant $k$-tuple whose entries are equal to $\alpha_k^{1/p}$. We have
\[
\Vert (h,f_{1})\Vert_g^p=\max_{0\le i \le k} \left( v_i+\alpha_k \sum_{j=1}^i w_j\right)
\le \max_{i \ge 0}  \left( v_i+ \alpha_i \sum_{j=1}^i w_j\right)=s<t,
\]
and 
\[
\Vert (f_{2},h)\Vert_g^p \ge\Vert f_{2} \Vert_{g}^p+\alpha_k\sum_{j=m+1}^{m+k} w_j\ge \Vert f_{2} \Vert_{g}^p+1. \qedhere
\]
\end{proof}

We will obtain  Proposition~\ref{Lemma:3}  by using the full power of Lemma~\ref{Lemma:1}.  Given $k\in\NN$ we will denote by $\vv[k]$ the positive constant-coefficient $k$-tuple whose norm in $g(\ww,p)$ is one, i.e.,
\[
\vv[k]=\frac{1}{(\sum_{j=1}^k w_j)^{1/p}}(\underbrace{1,\dots,1}_{k}).
\]

\begin{lemma}\label{Lemma:2}Let $1\le p<\infty$ and $\ww\in\WW$. Given $k_0\in\NN$, $t>1$ and a tuple $f$ with $ \Vert f\Vert_g< t$, there  is 
$k\ge k_0$ such that  $\Vert (\vv[k],f)\Vert_g < t$.
\end{lemma}
\begin{proof}Applying Lemma~\ref{Lemma:1} with $f_{1}=f$ and $f_{2}=0$ yields $k\ge k_0$ and a constant $k$-tuple $h$ verifying 
\[
\Vert (h, f)\Vert_g <t \text{ and } s:=\Vert h\Vert_g\ge 1.
\]
Notice that $h=s\, \vv[k]$. Using the $1$-unconditionality of the 
unit vector basis of $g(\ww,p)$ we obtain
\[
\Vert (\vv[k],f)\Vert_g = \Vert  s ^{-1}\, h , f)\Vert_g\le \Vert (h, f)\Vert_g<  t,
\]
as desired.
\end{proof}

Given a tuple $\kappa=(k_i)_{i=1}^n$ we put
\[
\vv[\kappa]=\vv[k_1,\dots, k_i,\dots k_n]=(\vv[k_1],\dots,\vv[k_i],\dots,\vv[k_n]).
\]
\begin{proposition}\label{Lemma:3}Let $1\le p<\infty$ and $\ww\in\WW$. Given
$1<t<\infty$,  $k\in \NN$, and $n\in\NN$, there is a  sequence
$\kappa=(k_i)_{i=1}^n$ such that $k_i\ge k$ for  $i=1$,\dots, $n$ and  with $\Vert\vv[\kappa]\Vert_g\le t$.
\end{proposition}

\begin{proof}Lemma~\ref{Lemma:2} allows us to recursively construct a sequence $(q_i)_{i=1}^\infty$ in $\NN$ such that $q_1=k$, $q_{i}\ge k$ for all $i\in\NN$, and 
\[
\Vert \vv[q_n,\dots, q_i,\dots,q_1]\Vert_g < t
\]
for all $n\in\NN$. To finish the proof we just need to take $\kappa=(q_{n+1-i})_{i=1}^n$  for a given $n\in\NN$.
\end{proof}

Given an increasing sequence  $\gamma=(k_n)_{n=1}^\infty$ of natural numbers, define 
$q(\gamma)=(q_n)_{n=0}^\infty$ by $q_n=\sum_{i=1}^{n}k_n$  and $P_\gamma\colon \FF^\NN\to \FF^\NN$ by
\[
f=(a_j)_{j=1}^\infty \mapsto P_\gamma(f)=
\left(\frac{(\sum_{j=1}^{q_n-q_{n-1}} w_j)^{1/p}}{ \sum_{j=1+q_{n-1}}^{q_n}  w_j} \sum_{j=1+q_{n-1}}^{q_n} a_j w_j\right)_{n=1}^\infty.
\]
From now on we will use the convention $\sum_{i=1}^0 a_i=0$.
\begin{lemma}\label{gwpTolp}let $\gamma$ be a sequence of  natural  numbers and $t\in(0,\infty)$. Assume that, if $q(\gamma)=(q_n)_{n=0}^\infty$, 
\[ 
\frac{ \sum_{j=1+q_{n-1}}^{q_n}  w_j}{\sum_{j=1}^{q_n-q_{n-1}} w_j}\ge t 
\]
for all $n\in\NN$. Then
$
\Vert P_\gamma \colon  g(\ww,p)\to \ell_p\Vert \le t^{-1/p}.
$
\end{lemma}

\begin{proof}By H\"older's inequality,  
\[
\left|\sum_{j=1+q_{n-1}}^{q_n} a_j w_j \right|^p\le  \left(\sum_{j=1+q_{n-1}}^{q_n} w_j\right)^{p-1} \sum_{j=1+q_{n-1}}^{q_n} |a_j |^p w_j,
\quad n\in\NN.
\]
Thus, if $f=(a_j)_{j=1}^\infty$,
\begin{align*}
\Vert P_\gamma(f)\Vert_p^p
&\le\sum_{n=1}^\infty   \frac{\sum_{j=1}^{q_n-q_{n-1}} w_j}{ \sum_{j=1+q_{n-1}}^{q_n}  w_j} 
\sum_{j=1+q_{n-1}}^{q_n} |a_j |^p w_j\\
&\le  \frac{1}{t} \sum_{n=1}^\infty  \sum_{j=1+q_{n-1}}^{q_n} |a_j |^p w_j \\
&= \frac{1}{t} \sum_{j=1}^\infty |a_j |^p w_j \\
&\le  \frac{1}{t} \Vert f \Vert_g^p.
\end{align*}
\end{proof}

In our route to prove Theorem~\ref{ComplementedBesov} we  need to revisit a result from \cite{AAW1}.

\begin{proposition}[cf. \cite{AAW1}*{Proposition 3.2}]\label{lpdominates} Let $1\le p<\infty$, $\ww\in\WW$ and $t\in(0,\infty)$. Let $\BB=(\yy_n)_{n=1}^\infty$ be a block basic sequence of the unit vector basis of $g(\ww, p)$ such that $\Vert \yy_n\Vert \le t $ for every $n\in\NN$. Then $\BB$ is $t$-dominated by  the unit vector basis of $\ell_p$.\end{proposition}
\begin{proof}Although \cite{AAW1}*{Proposition 3.2} tackles only the case when $\Vert y_n\Vert_g=1$,  its proof can be reproduced almost verbatim in this slightly more general setting.
\end{proof}

We are now in a position to complete the proof of Theorem~\ref{ComplementedBesov}.

\begin{proof}[Proof of Theorem~\ref{ComplementedBesov}] Let $\DD=\{(i,n) \in\NN^2 \colon 1\le i \le n\}$ and
$t=\sqrt{1+\varepsilon}$. Use Proposition~\ref{Lemma:3} and Lemma~\ref{Lemma:GlidingHumpWeight} to recursively construct $\kappa_n=(k_{i,n})_{i=1}^n$,  $n\in\NN$, verifying
 $\Vert \vv[\kappa_n]\Vert_g \le t$ for all $n\in\NN$ and
 \[
\frac{\sum_{j=1+m_n}^{m_n+k_{i,n}} w_j}{\sum_{j=1}^{k_{i,n}} w_j}\ge t^{-p}
\]
for all $(i,n)\in \DD$, where
$m_n=\sum_{r=1}^{n-1}\max_{1\le i \le k} k_{i,r}$ for $n\in\NN\cup\{0\}$.

For $n\in\NN$ and $0\le i \le n$ set
\[
m_{i,n}=\sum_{r=1}^{n-1} \sum_{d=1}^r k_{d,r}+ \sum_{d=1}^i k_{d,n}
\]
 and  for $(i,n)\in\DD$ put
\begin{itemize}
\item $J_{i,n}=\{n\in\NN \colon 1+m_{i-1,n}\le j \le m_{i,n}\}$,
\item $\yy_{i,n}= \left(\sum_{j=1}^{k_{i,n}} w_j\right)^{-1/p} \sum_{j\in J_{i,n}} \ee_j$,
\item
$
\yy_{i,n}^*((a_j)_{j=1}^\infty)= \left(\sum_{j=1}^{k_{i,n}} w_j\right)^{1/p} ( \sum_{j\in J_{i,n} }  w_j )^{-1} \sum_{j\in J_{i,n}} w_j a_j.
$
 \end{itemize}

Note that  $(J_{i,n})_{(i,n)\in\DD}$ is a partition $\NN$, and it is straightforward to check that $(\yy_{i,n}, \yy_{i,n}^*)_{(i,n)\in\DD}$ is a biorthogonal system. Thus, if we define
$P\colon \FF^\NN \to \prod_{n=1}^\infty \FF^n$ by
\[
 P(f) = \left(\left( \yy_{i,n}^*(f) \right)_{i=1}^n\right)_{n=1}^\infty, \quad f\in\FF^\NN,
\]
and $S\colon \prod_{n=1}^\infty \FF^n \to \FF^\NN $ by
\[
f=((a_{i,n})_{i=1}^n)_{n=1}^\infty\in  \prod_{n=1}^\infty \FF^n  \mapsto S(f)=\sum_{n=1}^\infty \sum_{i=1}^n a_{i,n}\, \yy_{i,n}
\]
we have $P\circ S=\Id_{\prod_{n=1}^\infty \FF^n}$. Thus, the proof will be over once we show that
\[
\max\{\Vert P \colon g(\ww,p) \to (\oplus_{n=1}^\infty \ell_\infty^n)_p\Vert,
\Vert S  \colon  (\oplus_{n=1}^\infty \ell_\infty^n)_p \to g(\ww,p) \Vert\} \le t.
\]

Given $\alpha=(i_n)_{n=1}^\infty$ with $1\le i_n \le n$,  let $\phi_\alpha\in\OO$ be defined by $\phi_\alpha(\NN)=K_\alpha:=\cup_{n=1}^\infty J_{i_n,n}$. Consider the operator
\[
V_\alpha\colon\FF^\NN\to \FF^\NN, \quad f=(a_j)_{j=1}^\infty \mapsto V_\alpha(f)=(a_{\phi_\alpha(j)})_{j=1}^\infty.
\] 
That is, $V_\alpha(f)$ is the sequence obtained by removing from $f$  its coefficients outside $K_\alpha$.
The $1$-subsymmetry of the unit vector basis of $g(\ww,p)$ yields $\Vert V_\alpha\colon g(\ww,p)\to g(\ww,p)\Vert \le 1$.  
Put $q_n(\alpha)=q_n:=\sum_{r=1}^n k_{i_r,r}$ for $n\in\NN\cup\{0\}$. Since $\ww$ is non-increasing,
\[
\frac{\sum_{j=1+q_{n-1}}^{q_n} w_j}{\sum_{j=1}^{k_{i_n,n}} w_j}
\ge
\frac{\sum_{j=1+m_n}^{m_n+k_{i_n,n}} w_j}{\sum_{j=1}^{k_{i_n,n}} w_j}
\ge t^{-p},\quad n\in\NN.
\]
Therefore, by Lemma~\ref{gwpTolp},
\[
\left(\sum_{n=1}^\infty |\yy_{i_n,n}^*(f)|^p\right)^{1/p}\le t \, \Vert V_\alpha (f)\Vert_g\le  t \, \Vert f\Vert_g, \quad f\in\FF^\NN.
\]
Taking the supremum on all possible choices of $(i_n)_{n=1}^\infty$ we obtain 
\[
\Vert P(f)\Vert_{ (\oplus_{n=1}^\infty \ell_\infty^n)_p}
=\left(\sum_{n=1}^\infty  (\max_{1\le i \le n}|\yy_{i,n}^*(f)|)^p\right)^{1/p}
\le  t\,  \Vert f\Vert_g, \quad f\in\FF^\NN.
\]
Let $\yy_n=\sum_{i=1}^n\yy_{i,n}$ for $n\in\NN$. We have that $(\yy_n)_{n=1}^\infty$ is a block basic sequence of the unit  vector system and that $\yy_n=\II_{m_{i-1,n}} (\vv[\kappa_n])$. Consequently, $\Vert \yy_n\Vert_g \le t$ for all $n\in\NN$.
If $f=((a_{i,n})_{i=1}^n)_{n=1}^\infty\in  \prod_{n=1}^\infty \FF^n$, invoking Proposition~\ref{lpdominates} and the $1$-unconditionality of the unit vector system we obtain 
\begin{align*}
\Vert S(f)\Vert_g 
&\le \left\Vert  \sum_{n=1}^\infty \left(\sup_{1\le i \le n} |a_{i,n}|\right)\left( \sum_{i=1}^n \yy_{i,n}\right)\right\Vert_g
\\
&=\left\Vert  \sum_{n=1}^\infty \left(\max_{1\le i \le n} |a_{i,n}|\right)  \yy_n \right\Vert_g
\\
&\le t \left(\sum_{n=1}^\infty  \left( \max_{1\le i \le n} |a_{i,n}|\right)^p\right)^{1/p}\\
&=t\,  \Vert f\Vert_{ (\oplus_{n=1}^\infty \ell_\infty^n)_p},
\end{align*}
as desired. \end{proof}

\section{Conditional bases in Garling sequence spaces}\label{ConditionalBases}

\noindent In 1964, Pe\l czy\'nski and Singer proved that every Banach space with a basis has a \textit{conditional} (i.e., not unconditional) basis \cite{PelSin1964}.  Thus in order to get a more accurate  information on a given space by means of conditional bases, one needs to restrict the discussion on their existence by imposing certain distinctive properties. 

One way to    specify a special property on  conditional bases is precisely by quantifying their conditionality.  In order to do that we consider the sequences $(k_m[\BB,\XX])_{m=1}^{\infty}$  and  $(L_m[\BB,\XX])_{m=1}^{\infty}$ defined   by
\begin{align*}
k_m[\BB]=k_m[\BB,\XX]&=\sup\left\{ \frac{ \Vert S_A(f)\Vert}{ \Vert f \Vert } \colon |A|\le m,  \, A\subseteq\NN\right\},\\
L_m[\BB]=L_m[\BB,\XX]&=\sup\left\{ \frac{ \Vert S_A(f)\Vert}{ \Vert f \Vert } \colon \supp(f)\subseteq [1,m], \, A\subseteq\NN\right\}.
\end{align*}
Indeed, since  a basis $\BB$ is unconditional if and only if $\sup_m L_m[\BB]<\infty$ or
$\sup_m k_m[\BB]<\infty$, the growth of any of those sequences can be interpreted as a measure of the conditionality of $\BB$.
The gauge $k_m[\BB]$  does not depend on the way in which the vectors of the basis are arranged, and so  is arguably more natural than  $L_m[\BB]$. However, the sequence $L_m[\BB]$  introduced in \cite{AAWo} 
has  shown to be in some settings a more accurate tool for studying conditional bases (see also \cite{AADK}).
Note that  $L_m[\BB]\le k_m[\BB]$.

For every basis $\BB$ in a Banach space $\XX$ one always has the estimate $k_m[\BB,\XX]\lesssim m$, for $m\in \NN$. Conversely, it is known (see \cite{AAWo}*{Theorem 3.5}) that  $\XX$ is not superreflexive if and only if there is a basic sequence $\BB^{\prime}$ in $\XX$ with $m\lesssim L_m[\BB^{\prime},\XX]$ for $m\in\NN$. 
Hence it is natural to wonder if, being non-superreflexive, $g(\ww,p)$ will possess not only a basic sequence but a basis of the whole space with this property.
 The answer is positive as we next show.

\begin{proposition}\label{ConditionalSchBases} Let $1\le p<\infty$ and $\ww\in\WW$. Then $g(\ww,p)$ has a (conditional) basis $\BB$ with $L_m[\BB]\approx m$ for $m\in\NN$.
\end{proposition}
\begin{proof}It is known that  the  \textit{summing system} $(\sss_j)_{j=1}^\infty$ given by
\[
\sss_j=\sum_{k=1}^j \ee_j 
\]
 is a basis for $c_0$  such that $\spn(\sss_j \colon 1\le j\le n)=\ell_\infty^n$ for all $n\in\NN$ with $L_m[\SSB]\approx m$ for $m\in\NN$
 (see, e.g., \cite{AADK}*{Lemma 4.9}). Hence, $\BB_0:=\bigoplus_{n=1}^\infty  (\sss_j)_{j=1}^{2^n}$ is a basis  for $\XX:=(\oplus_{n=1}^\infty \ell_\infty^{2^n})_p$ with $L_m[\BB_0]\approx m$ for $m\in\NN$
 (see \cite{AADK}*{Lemma 2.3}). Then $\BB:=\BB_0\oplus (\ee_j)_{j=1}^\infty$ is a basis for $\YY=\XX\oplus g(\ww,p)$ with $L_m[\BB]\approx m$ for $m\in\NN$ (see \cite{AADK}*{Lemma 2.2}). Finally, since 
 $\XX\oplus (\oplus_{n=1}^\infty \ell_\infty^{n})_p \approx (\oplus_{n=1}^\infty \ell_\infty^{n})_p$
 (see, e.g.,  \cite{AA2017}*{Appendix 4.1}), Theorem~\ref{ComplementedBesov} yields $\YY\approx g(\ww,p)$.
\end{proof}
 
 Now we will look into conditional bases in Garling sequence spaces that have  some special features in relation to the optimality of the greedy algorithm.  For the convenience of the reader we recall the relevant concepts from approximation theory, thus making our exposition self-contained.

Let $\BB=(\xx_{n})_{n=1}^{\infty}$ be a basis for a Banach space $\XX$. A finite set $G\subseteq \NN$ is said to be a \textit{greedy set} for 
$f=\sum_{n=1}^\infty a_n \, \xx_n \in\XX$ if  $|a_n|\ge |a_j|$ whenever $n\in G$ and $j\in\NN\setminus G$.  A \textit{greedy sum} will be a coordinate projection on a greedy set.
A basis $\BB$ is said to be \textit{almost greedy} if 
the greedy sums (essentially) provide  the optimal approximations amongst coordinate projections, that is, there is a constant $C<\infty$ such that whenever  $G$ is a greedy set for $f\in \XX$ and $|G|=|A|$,
\[
\Vert f-S_G(f)\Vert \le C \Vert f - S_A(f)\Vert.
\]
Almost greedy bases enjoy the property  of being \textit{democratic}  (see \cite{DKK2003}*{Theorem 3.3}), i.e.,
 there is a sequence  $(\lambda_m)_{m=1}^\infty$ such that for any finite subset $A$ of $\NN$,
\[
\left\Vert \sum_{j\in A} \varepsilon_j \xx_j\right\Vert \approx \lambda_{|A|}.
\]
In this case, $(\lambda_m)_{m=1}^\infty$ is equivalent to the \textit{fundamental function} $\varphi_m[\BB,\XX]$
of $\BB$ defined by
\[
\varphi_m[\BB,\XX]=\sup_{|A|\le m} \left\Vert \sum_{j\in A} \xx_j\right\Vert, \quad m\in\NN.
\]
Note that the unit vector basis $\EE=(\ee_j)_{j=1}^\infty$ of $g(\ww,p)$  verifies 
\[
\left\Vert \sum_{j\in A} \ee_j\right\Vert_g=\left(\sum_{j=1}^{|A|} w_j\right)^{1/p}
\] for all  $A\subseteq\NN$ finite and so 
\[
\varphi_m[\EE,g(\ww,p)]=\left(\sum_{j=1}^m w_j\right)^{1/p},\quad m\in\NN.\]
 
When a basis $\BB$ of a Banach space $\XX$ is almost greedy, the size of the members of the sequence $(k_m[\BB, \XX])_{m=1}^{\infty}$ is controlled by a slowly growing function to the extent that  (see \cite{DKK2003}*{Lemma 8.2}) 
\begin{equation}\label{DKKInequality}
k_m[\BB]\lesssim \log m, \quad m\ge 2.
\end{equation}
Moreover, by \cite{AAGHR2015}*{Theorem 1.1} this inequality is optimal only if $\XX$ is not superreflexive.

We close with a new addition to the subject   of  finding (non-super\-reflexive) spaces possessing almost greedy conditional bases for which the estimate \eqref{DKKInequality} is optimal, i.e., $k_m[\BB]\approx \log m$ for $m\ge 2$. This  topic  was  initiated by Garrig\'os et al.\ in \cite{GHO2013} and has been given continuity  through several papers and authors
(see \cites{GW2014,AAGHR2015,AAWo,BBGHO2018,AADK}).
  
 The proof of Theorem~\ref{AGC} leans on regularity properties of weights  whose definitions we refresh. 
 
 A weight $\ww=(w_{j})_{n=1}^{\infty}$  is said to be \textit{bi-regular} if both $\ww$ and its 
 \textit{conjugate weight} $\ww^*=(1/(j w_j))_{j=1}^\infty$  are regular.
 Following  \cite{DKKT2003},  a weight $(\lambda_m)_{m=1}^\infty$  is said to have the \textit{lower regularity property} (LRP for short) if there is a positive integer $b$ such that
 \[
 2\lambda_m \le \lambda_{bm}. 
\quad
 m\in\NN,
 \] 
A weight $(\lambda_m)_{m=1}^\infty$ is said to have the \textit{upper regularity property} (URP for short) if there is an integer $b\ge 3$ such that
\[
 \lambda_{b m}\le \frac{b}{2} \lambda_m,\quad m\in\NN.
\] 

\begin{theorem}\label{AGC}Let  $1\le p<\infty$ and $\ww\in\WW$.
\begin{itemize}
\item[(a)] There is an almost greedy basis $\BB$ for $g(\ww,p)$ with fundamental function equivalent to $(m^{1/p})_{m=1}^\infty$ such that
$L_m[\BB]\approx \log m$ for $m\ge 2$.

\item[(b)] If the weight $\ww$ is bi-regular there is an almost greedy basis $\BB$ for $g(\ww,1)$ with fundamental function equivalent to 
$(\sum_{j=1}^m w_j)_{m=1}^\infty$ such that
$L_m[\BB]\approx \log m$ for $m\ge 2$.

\item[(c)] If $\ww$ is regular and $p>1$, there is an almost greedy basis $\BB$ for $g(\ww,p)$ with fundamental function equivalent to the sequence
$((\sum_{j=1}^m w_j)^{1/p})_{m=1}^\infty$ such that
$L_m[\BB]\approx \log m$ for $m\ge 2$.

\end{itemize}
\end{theorem}

\begin{proof} Applying \cite{AADK}*{Theorem 4.1} with $\Sym=\ell_p$, and taking into account
Theorem~\ref{ConditionalSchBases}, gives a basis $\BB$ as claimed in (a) for the  Banach space $\ell_p \oplus g(\ww,p)\approx g(\ww,p)$.

Assume that $\ww$ is regular. Then, by \cite{AAW2}*{Proposition 2.5},  its primitive weight $(\sum_{j=1}^m w_j)_{m=1}^\infty$ has the LRP. Therefore $((\sum_{j=1}^m w_j)^{1/p})_{m=1}^\infty$ also has the LRP for any $1\le p<\infty$. Note also that $m w_m \approx \sum_{j=1}^n w_j$ and so the conjugate weight $\ww^*$ is equivalent to $1/(\sum_{j=1}^m w_j)_{m=1}^\infty$. Consequently, \cite{AA2015}*{Lemma 2.12 (iii)} yields that  $(\sum_{j=1}^m w_j)_{m=1}^\infty$ has the URP whenever $\ww^*$ is bi-regular. Since
$m^{-1}(\sum_{j=1}^m w_j)_{m=1}^\infty$ is non-increasing,  in the case when $p>1$, \cite{AA2015}*{Lemma 2.12 (ii)}  yields 
$((\sum_{j=1}^m w_j)^{1/p})_{m=1}^\infty$ has the URP.

Under the assumptions in  both  (b) and (c),  applying  \cite{AADK}*{Remark 4.2} with
$\Sym=g(\ww,p)$ gives a basis as desired for the Banach space $g(\ww,p) \oplus g(\ww,p)\approx g(\ww,p)$.
\end{proof}




\begin{bibsection}
\begin{biblist}

\bib{AA2015}{article}{
 author={Albiac, F.},
 author={Ansorena, J.~L.},
 title={Lorentz spaces and embeddings induced by almost greedy bases in
 Banach spaces},
 journal={Constr. Approx.},
 volume={43},
 date={2016},
 number={2},
 pages={197--215},
}


\bib{AA2017}{article}{
 author={Albiac, F.},
 author={Ansorena, J.~L.},
 title={Isomorphic classification of mixed sequence spaces and of Besov
 spaces over $[0,1]^d$},
 journal={Math. Nachr.},
 volume={290},
 date={2017},
 number={8-9},
 pages={1177--1186},
}

\bib{AADK}{article}{
 author={Albiac, F.},
 author={Ansorena, J.~L.},
 author={Dilworth, S.~J.},
 author={Kutzarova, Denka},
 title={Building  highly conditional almost greedy and quasi-greedy bases in Banach spaces},
 journal={J. Funct. Anal.},
 date={2018},
 doi={10.1016/j.jfa.2018.08.015},
 }
 
 \bib{AAGHR2015}{article}{
 author={Albiac, F.},
 author={Ansorena, J.~L.},
 author={Garrig{\'o}s, G.},
 author={Hern{\'a}ndez, E.},
 author={Raja, M.},
 title={Conditionality constants of quasi-greedy bases in super-reflexive
 Banach spaces},
 journal={Studia Math.},
 volume={227},
 date={2015},
 number={2},
 pages={133--140},
}

\bib{AALW}{article}{
 author={Albiac, F.},
 author={Ansorena, J.~L.},
 author={Leung, D.},
 author={Wallis, B.},
 title={Optimality of the rearrangement inequality with applications to
 Lorentz-type sequence spaces},
 journal={Math. Inequal. Appl.},
 volume={21},
 date={2018},
 number={1},
 pages={127--132},
 }

\bib{AAW1}{article}{
author={Albiac, F.},
author={Ansorena, J.L.},
author={Wallis, B.},
title={Garling sequence spaces},
journal={J. London Math. Soc.},
volume={98},
 date={2018},
 number={2},
 pages={204--222}, }
 
 \bib{AAW2}{article}{
 author={Albiac, F.},
 author={Ansorena, J.~L.},
 author={Wallis, B.},
 title={1-greedy renormings of Garling sequence spaces},
 journal={J. Approx. Theory},
 volume={230},
 date={2018},
 pages={13--23},
}

\bib{AAWo}{article}{
 author={Albiac, F.},
 author={Ansorena, J.~L.},
 author={Wojtaszczyk, P.},
 title={Conditionality constants of quasi-greedy bases in non-superreflexive Banach space},
 journal={Constr. Approx.},
 doi= {10.1007/s00365-017-9399-x},
}

\bib{AlbiacKalton2016}{book}{
 author={Albiac, F.},
 author={Kalton, N.~J.},
 title={Topics in Banach space theory, 2nd revised and updated edition},
 series={Graduate Texts in Mathematics},
 volume={233},
 publisher={Springer International Publishing},
 date={2016},
 pages={xx+508},
 }
 
 \bib{Altshuler1975}{article}{
 author={Altshuler, Z.},
 title={Uniform convexity in Lorentz sequence spaces},
 journal={Israel J. Math.},
 volume={20},
 date={1975},
 number={3-4},
 pages={260--274},
 }

 \bib{BBGHO2018}{article}{
author={Bern\'a, P.},
 author={Blasco, O.},
 author={Garrig{\'o}s, G.},
 author={Hern{\'a}ndez, E.},
 author={Oikhberg, T.},
 title={Embeddings and Lebesgue-type inequalities for the greedy algorithm in Banach spaces},
 journal={Constr. Approx.},
 date={2018},
 doi={10.1007/s00365-018-9415-9},
 }
 
 \bib{ACL73}{article}{
   author={Altshuler, Z.},
   author={Casazza, P.~G.},
   author={Lin, B.~L.},
   title={On symmetric basic sequences in Lorentz sequence spaces},
   journal={Israel J. Math.},
   volume={15},
   date={1973},
   pages={140--155},
}

\bib{CL74}{article}{
   author={Casazza, P.G.},
   author={Lin, B.~L.},
   title={On symmetric basic sequences in Lorentz sequence spaces. II},
   journal={Israel J. Math.},
   volume={17},
   date={1974},
   pages={191--218},
}

 \bib{DKK2003}{article}{
 author={Dilworth, S.~J.},
 author={Kalton, N.~J.},
 author={Kutzarova, Denka},
 title={On the existence of almost greedy bases in Banach spaces},
 journal={Studia Math.},
 volume={159},
 date={2003},
 number={1},
 pages={67--101},
 }

\bib{DKKT2003}{article}{
 author={Dilworth, S.~J.},
 author={Kalton, N.~J.},
 author={Kutzarova, Denka},
 author={Temlyakov, V.~N.},
 title={The thresholding greedy algorithm, greedy bases, and duality},
 journal={Constr. Approx.},
 volume={19},
 date={2003},
 number={4},
 pages={575--597},
 }

\bib{Garling1968}{article}{
 author={Garling, D.~J.~H.},
 title={Symmetric bases of locally convex spaces},
 journal={Studia Math.},
 volume={30},
 date={1968},
 pages={163--181},
 }

\bib{GHO2013}{article}{
 author={Garrig{\'o}s, G.},
 author={Hern{\'a}ndez, E.},
 author={Oikhberg, T.},
 title={Lebesgue-type inequalities for quasi-greedy bases},
 journal={Constr. Approx.},
 volume={38},
 date={2013},
 number={3},
 pages={447--470},
 }

\bib{GW2014}{article}{
 author={Garrig{\'o}s, G.},
 author={Wojtaszczyk, P.},
 title={Conditional quasi-greedy bases in Hilbert and Banach spaces},
 journal={Indiana Univ. Math. J.},
 volume={63},
 date={2014},
 number={4},
 pages={1017--1036},
 }
 
 \bib{PelSin1964}{article}{
   author={Pe\l czy\'nski, A.},
   author={Singer, I.},
   title={On non-equivalent bases and conditional bases in Banach spaces},
   journal={Studia Math.},
   volume={25},
   date={1964/1965},
   pages={5--25},
}
 
\end{biblist}
\end{bibsection}
\end{document}